\documentclass[12pt]{amsart}
\usepackage{epsfig,color}

\makeatother

\theoremstyle{definition}

\def\fnum{equation} 
\newtheorem{Thm}[\fnum]{Theorem}
\newtheorem{Cor}[\fnum]{Corollary}

\newtheorem{Lem}[\fnum]{Lemma}

\newtheorem{Pro}[\fnum]{Proposition}

\numberwithin{equation}{section}

\newcommand{\Vol}{{\text{Vol}}}

\newcommand{\nn}{{\bf{n}}}

\def\RR{{\bold R}}

\def\SS{{\bold S}}

\newcommand{\dv}{{\text {div}}}
\newcommand{\e}{{\text {e}}}

\newcommand{\R}{{\bf{R}}}

\newcommand{\cL}{{\mathcal{L}}}

\newcommand{\eqr}[1]{(\ref{#1})}
\newcommand{\eps}{\epsilon}

\title[Rigidity of generic singularities of mean curvature flow]{Rigidity
  of generic singularities of mean curvature flow}

\author{Tobias Holck Colding}%
\address{MIT, Dept. of Math.\\
77 Massachusetts Avenue, Cambridge, MA 02139-4307.}
\author{Tom Ilmanen}%
\address{Departement Mathematik, ETH Zentrum, 8092 Z\"urich, Switzerland.}
\author{William P. Minicozzi II}%

\thanks{The first and third authors
were partially supported by NSF Grants DMS  11040934, DMS 1206827,  and NSF FRG grants DMS 
 0854774 and DMS 0853501}

\email{colding@math.mit.edu, tom.ilmanen@math.ethz.ch, and minicozz@math.mit.edu}

\begin{document}

\maketitle

\begin{abstract}
  Shrinkers are special solutions of mean curvature flow (MCF)
  that evolve by rescaling and model the singularities.  While there
  are infinitely many in each dimension, \cite{CM1} showed that the
  only generic are round cylinders $\SS^k\times \RR^{n-k}$.  We prove
  here that round cylinders are rigid in a very strong sense.  Namely,
  any other shrinker that is sufficiently close to one of them on
  a large, but compact, set must itself be a round cylinder.

  To our knowledge, this is the first general rigidity theorem for
  singularities of a nonlinear geometric flow.  We expect that the
  techniques and ideas developed here have applications to other
  flows.

  Our results hold in all dimensions and do not require any a priori smoothness.
\end{abstract}

\section{Introduction}

The mean curvature flow is an evolution equation where a hypersurface
evolves over time by locally moving in the direction of steepest
descent for the volume element.

A hypersurface $\Sigma \subset \RR^{n+1}$ is said to be a
{\emph{self-similar shrinker}}, or just {\emph{shrinker}}, if it
is the $t=-1$ time-slice of a mean curvature flow (``MCF'') moving by
rescalings.{\footnote{This means that the time $t$ slice of the MCF is
    given by $\sqrt{-t} \, \Sigma$.  See \cite{A}, \cite{Ch},
    \cite{KKM}, and \cite{N} for examples of shrinkers.}}  Being
self-similar is easily seen to be equivalent to being stationary for
the Gaussian surface area.

By a monotonicity formula of Huisken and an argument of Ilmanen and
White, blow-ups of singularities of a MCF are  shrinkers.
The only generic shrinkers are round cylinders by \cite{CM1}.

\vskip2mm
Our main theorem is the following rigidity or uniqueness theorem for   cylinders:

\begin{Thm} \label{t:uniq1} Given $n$, $\lambda_0$ and $C$, there
  exists $R=R(n,\lambda_0 , C)$ so that if $\Sigma^n \subset
  \RR^{n+1}$ is a shrinker with entropy $\lambda (\Sigma) \leq
  \lambda_0$ satisfying
  \begin{itemize}
  \item[($\dagger$)] $\Sigma$ is smooth in $B_R$ with $H\geq 0$ and $|A| \leq
    C$ on $ B_R \cap \Sigma$,
  \end{itemize}
  then $\Sigma$ is a generalized cylinder $\SS^k \times \RR^{n-k}$ for
  some $k \leq n$.
\end{Thm}

Smooth will always mean smooth and embedded. \vskip1mm

The entropy $\lambda $ is the supremum of the Gaussian surface areas
of $\Sigma$ over all centers and scales.  When $\Sigma$ is a shrinker,
there is no need to take a supremum as the entropy is always achieved
by the standard Gaussian surface area (see section $7$ of \cite{CM1}).

In particular, Theorem \ref{t:uniq1} implies that a shrinker that is
close to a cylinder in a sufficiently large ball must be isometric to
the cylinder.

\vskip2mm 
We will say that a singular point is {\emph{cylindrical}} if
there is at least one tangent flow that is a multiplicity one
generalized cylinder $\SS^k \times \RR^{n-k}$.  We will prove the
following gap theorem for singularities in a neighborhood of a
cylindrical singular point:

\begin{Thm} \label{t:main} Let $M_t$ be a  MCF in $\RR^{n+1}$.
  Each cylindrical singular point has a space-time neighborhood where
  every non-cylindrical singular point has entropy at least $\eps =
  \eps (n) > 0$ below that of the cylinder.
\end{Thm}

As a corollary of the theorem, we get uniqueness of type for
cylindrical tangent flows:

\begin{Cor} 
  \label{c:axis}
  If one tangent flow at a singular point of a mean curvature flow is
  a multiplicity one cylinder, then they all are.
\end{Cor}

By definition, a tangent flow is the limit of a sequence of rescalings
at the singularity, where the convergence is on compact subsets.  For
this reason, it is essential for applications of Theorem
\ref{t:uniq1}, like Corollary \ref{c:axis}, that Theorem \ref{t:uniq1}
only requires closeness on a fixed compact set.
 
These results are the first general uniqueness theorems for tangent
flows to a geometric flow at a non-compact singularity.  Some special
cases for MCF were previously analyzed assuming either some sort of
convexity or that the hypersurface is a surface of rotation; see
\cite{H1}, \cite{H2}, \cite{HS}, \cite{W1}, \cite{SS}, \cite{AAG}, and
section $3.2$ in the book \cite{GGS}.  In contrast, uniqueness for
blowups at compact singularities is better understood; cf.\ \cite{AA}
and \cite{Si}.

\vskip2mm In each dimension, the sphere has the lowest entropy among
closed shrinkers, but in higher dimensions there are smooth noncompact
shrinkers, very close to Simons cones, with entropy below the
cylinders; see \cite{CIMW} or \cite{I1} for details.  The results of
this paper may be compared and contrasted with Brakke's regularity
theorem, \cite{B}, which shows not only that the hyperplane isolated among
shrinkers, but there is a gap to the entropy of all other shrinkers.

\subsection{Rigidity without smoothness}
 
The rigidity theorem holds even when the shrinker is not required to 
be smooth outside of the ball of radius $R$.  This is important in 
applications, including Theorem \ref{t:main} and Corollary
\ref{c:axis}.
  
Smooth shrinkers are stationary for Gaussian surface area (called the
$F$ functional in \cite{CM1}). Accordingly, weak solutions of the
shrinker equation can be defined as $F$-stationary $n$-dimensional
integer rectifiable varifolds.  We will refer to these simply as
$F$-stationary varifolds.  Theorem \ref{t:uniq1} holds for
$F$-stationary varifolds that are smooth in $B_R$ with the given
estimates.
 
  \subsection{Outline of proof}
  
  The proof of Theorem \ref{t:uniq1} uses iteration and improvement. 
  Roughly speaking, the theorem assumes that the shrinker is   cylindrical on some  large scale.   
  The iterative step then shows that it is  cylindrical on an even larger scale,
  but with some loss in the estimates.  The improvement step then comes back and says 
  that there was actually no loss if the scale is large enough.  Applying these two steps 
  repeatedly gives that 
 the shrinker is   roughly cylindrical on all scales, which will easily give the theorem.  
  We will make this outline more precise in Section \ref{s:ingreds}.

 \vskip2mm
 
Corollary \ref{c:axis}  shows that if one tangent flow at a singular point of a MCF is a multiplicity one cylinder, then all are. However, it leaves open the possibility that the direction of the axis (the $\RR^k$ factor) depends on the sequence of rescalings. This has since been settled in \cite{CM3} where it was shown that even the axis is independent of the sequence of rescalings.  The proof given there, in particular, the first Lojasiewicz type inequality there, has its roots in   ideas and inequalities of this paper and in fact implicitly uses that cylinders are isolated among shrinkers
by Theorem \ref{t:uniq1}.    The  results here and in \cite{CM3} were also used in 
  \cite{CM4}, where new rectifiability results for the singular set were obtained.

\vskip2mm
  
 The results of this paper were discussed in \cite{CMP}.

\section{Notation and background from \cite{CM1}}

We begin by recalling the classification of smooth, embedded mean
convex shrinkers in arbitrary dimension from theorem $0.17$ in
\cite{CM1}:

  \begin{Thm}[\cite{CM1}] \label{t:huisken} $\SS^k\times \RR^{n-k}$
    are the only smooth complete embedded shrinkers without boundary,
    with polynomial volume growth, and $H \geq 0$ in $\RR^{n+1}$.
  \end{Thm}

  The $\SS^k$ factor in Theorem \ref{t:huisken} is round, centered at
  $0$, and has radius $\sqrt{2k}$; we allow the possibilities of a
  sphere ($n-k = 0$) or a hyperplane (i.e., $k=0$), although Brakke's
  theorem rules out the multiplicity one hyperplane as a tangent flow
  at a singular point.

  The classification of smooth embedded shrinkers with $H \geq 0$
  began with \cite{H1}, where Huisken showed that round spheres are
  the closed ones. In a second paper, \cite{H2}, Huisken showed that
  the generalized cylinders $\SS^k\times \RR^{n-k}\subset \RR^{n+1}$
  are the open ones with polynomial volume growth and $|A|$ bounded.
  Theorem $0.17$ in \cite{CM1} completed the classification by
  removing the $|A|$ bound.
 
 \subsection{Notation}

Let $\Sigma \subset \RR^{n+1}$ be a hypersurface, $\nn$ its unit
normal, $\Delta$ its Laplace operator, $A$ its second fundamental
form, $H = \dv_{\Sigma} \, \nn$ its mean curvature, $x$ is the
position vector, and let $v^T$ denote the tangential projection of a
vectorfield $v$ onto $\Sigma$.

It is easy to see that a hypersurface $\Sigma$ is a shrinker if it satisfies the equation
\begin{equation}	\label{e:ss}
	H = \frac{ \langle x , \nn \rangle }{2} \, .
\end{equation}
 We will use the  operators $\cL$ and $L$ on shrinkers from \cite{CM1} defined by
\begin{align}
	\cL &= \Delta - \frac{1}{2} \nabla_{x^T} \, , \\
	L &= \cL + |A|^2 + \frac{1}{2} \, .
\end{align}
With our convention on $H$, 
a one-parameter family of hypersurfaces  $\Sigma_t \subset \RR^{n+1}$  flows by mean curvature if
\begin{equation}
	\left( \partial_t x \right)^{\perp} =  -H\,\nn \, .
\end{equation}

It is convenient to recall the family of functionals on the space of
hypersurfaces given by integrating Gaussian weights with varying
centers and scales.  These are often referred to as Gaussian surface
areas.  For $t_0>0$ and $x_0\in \RR^{n+1}$, define $F_{x_0,t_0}$ by
\begin{align}
  F_{x_0 , t_0} (\Sigma) = (4\pi t_0)^{-n/2} \, \int_{\Sigma} \,
  \e^{-\frac{|x- x_0|^2}{4t_0}} \, d\mu \, .
\end{align}
We will think of $x_0$ as being the point in space that we focus on
and $\sqrt{t_0}$ as being the scale. Write $F=F_{0,1}$.
      
The entropy is the supremum over all Gaussians and is given by 
\begin{equation}
  \lambda (\Sigma)=\sup_{x_0 , t_0}   \, \,  F_{x_0,t_0} (\Sigma)  \, .
\end{equation}
Here the supremum is over all $t_0>0$ and $x_0\in \RR^{n+1}$.  The
entropy is invariant under dilation and rigid motions and, as a
consequence of a result of Huisken, \cite{H1}, is monotone
nonincreasing under both MCF and rescaled MCF.

Note that both the $F$-functionals and 
the entropy are defined for rectifiable varifolds.

\section{Key ingredients for the rigidity of the
  cylinder} \label{s:ingreds}

In this section, we will prove the main theorem using the iterative
step (Proposition \ref{p:step1} below) and the improvement step
(Proposition \ref{p:step2} below).  These propositions will be proven
in Sections \ref{s:step1} and \ref{s:step2}, respectively.  We will
use the main theorem to prove Theorem \ref{t:main} and Corollary
\ref{c:axis} in subsection \ref{ss:proofs}.

\vskip2mm 
Throughout this section $\Sigma \subset \RR^{n+1}$ is an
$n$-dimensional $F$-stationary varifold.  On a first reading, the
reader may prefer to ignore the technicalities that arise because
$\Sigma$ is not assumed to be smooth.  The extra difficulties dealing
with non-smoothness are very minor and easy to overcome.

Smooth will always mean smooth and embedded.
\vskip1mm

\begin{Pro} \label{p:step1} Given $\lambda_0 < 2$ and $n$, there exist
  positive constants $R_0$, $\delta_0$, $C_0$ and $\theta$ so that if
  $\lambda (\Sigma) \leq \lambda_0$, $R \geq R_0$, and
  \begin{itemize}
  \item $B_R \cap \Sigma$ is smooth with $H\geq 1/4$ and $|A| \leq
    2$,
  \end{itemize}
  then $B_{(1+\theta)R} \cap \Sigma$ is smooth with $H \geq \delta_0 $
  and $|A| \leq C_0$.
\end{Pro}

From now on, $\delta_0$ and $C_0$ will be given by the previous proposition.

\begin{Pro} \label{p:step2} Given $n$, $\lambda_0 > 0$, $\delta_0 > 0
  $ and $C_0$, there exists $R_1$ so that if $\lambda (\Sigma) \leq
  \lambda_0$, $R \geq R_1$, and
  \begin{itemize}
  \item $B_{R} \cap \Sigma$ is smooth with $H\geq \delta_0 $ and
    $|A| \leq C_0$,
  \end{itemize}
  then $H\geq 1/4$ and $|A| \leq 2$ on $B_{R-3} \cap \Sigma$.
\end{Pro}

When we apply these iteratively, it will be important that Proposition
\ref{p:step1} extends the scale of the ``cylindrical region'' by a
factor greater than one, while Proposition \ref{p:step2} only forces
one to come in by a constant amount to get the improvement.  This
makes the iteration work as long as the initial scale is large enough.

\subsection{The proof of the rigidity theorem}
 
 We will see that the main theorem follows from the following proposition, where we also assume a positive
 lower bound for $H$ and an upper bound for the entropy that is less than two.
 
 \begin{Pro}		\label{p:uniq2}
Suppose that $\Sigma$ satisfies the hypotheses of  Theorem \ref{t:uniq1}.
If, in addition, we have the stronger assumptions that $\lambda_0 < 2$
and $H \geq \delta > 0$ on $B_{R-1}$ where $R=R(n, \lambda_0 , C , \delta)$,
then   $\Sigma$ is a generalized cylinder $\SS^k \times \RR^{n-k}$ for some $k \leq n$.
\end{Pro}

\begin{proof}[Proof of Proposition \ref{p:uniq2} using Propositions
  \ref{p:step1} and \ref{p:step2}]
  This follows by applying first Proposition \ref{p:step1} then
  Proposition \ref{p:step2} and repeating this.
\end{proof}

We will use the following elementary lemma in the proof of the main theorem.

\begin{Lem} \label{l:elem} 
There exists $C$ depending only on $n$ such
  that if $R>0$ and $\mu$ is a measure with $\lambda(\mu)\le\lambda_0$,
  then
  \begin{align} \label{e:elem1}
    F(\mu\lfloor(\RR^{n+1}\setminus B_R))\le C \e^{-R^2/8}\lambda_0 \, .
  \end{align} 
  In particular, if $\mu$ is the mass measure of an $F$-stationary
  varifold $\Sigma$, then
  \begin{align} \label{e:elem2}
    \left| \lambda (\Sigma) - F(\Sigma\cap B_R)\right| \leq C
      \e^{-R^2/8}\lambda_0 \, .
  \end{align}
\end{Lem}

\begin{proof}
Write $\rho_{x,t}(z):=\e^{-|x-z|^2/4t}/( 4\pi t)^{n/2}$. Compute
\begin{align*}
  F_{0,1}(\mu\lfloor(\RR^{n+1}\setminus B_R))\le\sup_{|z|\geq
    R}\frac{\rho_{0,1}(z)}{\rho_{0,2}(z)}\,\,F_{0,2}(\Sigma\setminus
  B_R)\leq 2^{n/2}\e^{-R^2/8}\lambda_0\,.
\end{align*}
When $\Sigma$ is $F$-stationary, \cite{CM1} gives that $\lambda
(\Sigma) = F_{0,1} (\Sigma)$, yielding the second statement. 
\end{proof}
 
We will also use the following compactness theorem for sequences of $F$-stationary varifolds:

\begin{Lem} \label{l:cpns} 
  Let $\Sigma_i \subset \RR^{n+1}$ be a
  sequence of $F$-stationary varifolds with $\lambda (\Sigma_i) \leq
  \lambda_0$ and
  \begin{align}
    B_{R_i} \cap \Sigma_i {\text{ is smooth with }}   \, |A| \leq C \, ,
  \end{align}
  where $R_i \to \infty$.  Then there exists a subsequence $\Sigma_i'$
  that converges smoothly and with multiplicity one to a complete
  embedded shrinker $\Sigma$ with $|A| \leq C$ and
  \begin{align} \label{e:lamcv} \lim_{i \to \infty} \, \lambda
    (\Sigma_i') = \lambda (\Sigma) \, .
  \end{align}
\end{Lem}

\begin{proof}
  All of this discussion is done in great detail in \cite{CM2} for a
  similar case, so we will sketch the argument here.

  Combining the a priori $|A|$ bound with elliptic estimates for the
  shrinker equation, the Arzela-Ascoli theorem, the strong maximum
  principle, and a diagonal argument, we get that a subsequence
  converges in $C^{2,\alpha}$ with finite multiplicity to a smooth
  embedded shrinker $\Sigma$ with $\lambda (\Sigma) \leq \lambda_0$.

  We argue as in \cite{CM2} to see why the multiplicity must be one.
  Namely, proposition $3.2$ in \cite{CM2} gives that $\Sigma$ is
  $L$-stable if the multiplicity is greater than
  one.{\footnote{Proposition $3.2$ in \cite{CM2} is stated for
      surfaces (i.e., $n=2$) but the proof applies in arbitrary
      dimension.}}  On the other hand, by \cite{CM1} (see theorem
  $0.5$ in \cite{CM2}), there are no complete $L$-stable shrinkers
  with polynomial volume growth, so the limit must have been
  multiplicity one.

  Finally, \eqr{e:lamcv} now follows from \eqr{e:elem2}.
\end{proof}

We will now give the proof of the main rigidity theorem. Though nearly
all of the work involves smooth computations, the argument actually
proves the general version where $\Sigma$ is $F$-stationary and need
not be smooth everywhere.  The proof uses Proposition \ref{p:uniq2}
which relies on the two key ingredients, Propositions \ref{p:step1}
and \ref{p:step2}, which will be proven in the next two sections.

\begin{proof}[Proof of Theorem \ref{t:uniq1} using Proposition
  \ref{p:uniq2}]
  By the compactness of Lemma \ref{l:cpns} and the classification of
  complete embedded mean convex shrinkers in \cite{H1} and \cite{H2},
  we can assume that $\Sigma$ is smoothly close to $\SS^k \times
  \RR^{n-k}$ in $B_{R_1}$ for some $k\in\{0,\ldots,n\}$, where $R_1$
  can be taken as large as we wish. Note that $F_{0,1} ( \SS^k \times
  \RR^{n-k}) \leq \sqrt {2\pi/\e} \approx 1.52$ by \cite{S}.  Then,
  using the closeness to the cylinder and \eqr{e:elem2}, we can arrange that
  \begin{align}
    \lambda (\Sigma) < 1.6 \, .
  \end{align}
  We now consider two cases depending on $k$:
  \begin{itemize}
  \item When $k=0$, then \eqr{e:elem2} allows us to make $\lambda
    (\Sigma)$ as close to $1$ as we wish, so Brakke's theorem
    (\cite[Thm 6.11]{B}; cf.\ \cite{W2}) gives that $\Sigma$ is a hyperplane.
  \item When $k>1$, we get that $H$ is approximately $\sqrt{k/2}$ on
    $B_{R_1} \cap \Sigma$ and the theorem follows from Proposition
    \ref{p:uniq2}.
  \end{itemize}
\end{proof}

\subsection{The proofs of Theorem \ref{t:main} and Corollary
\ref{c:axis}} \label{ss:proofs}

Theorem \ref{t:main} will follow from a slightly more general form of
the rigidity theorem for the cylinder and a compactness result for
mean curvature flows.

To make this precise, we will define a distance $d_V$ between Radon
measures on $\RR^{n+1}$ with the Gaussian weight.  Namely, let $f_n$
be a countable dense subset of the unit ball in the space of continuous
functions with compact support and and define
\begin{align}
  d_V (\mu_1 , \mu_2) = \sum_k \, 2^{-k} \, \left| \int f_k \,
    \e^{-|x|^2/4} \, d\mu_1 - \int f_k \, \e^{-|x|^2/4} \,
    d\mu_2 \right| \, .
\end{align}
It is then easy to see that $d_V$ is a metric on the space of Radon
measures satisfying $F(\mu)<\infty$ and $\mu_j \to
\mu$ in the standard weak topology if and only if $d_V (\mu_j , \mu) \to
0$.
  
\begin{Cor} \label{c:uniq2} Given $n$ and $\lambda_0$, there exists
  $\kappa > 0$ so that if $\Sigma^n \subset \RR^{n+1}$ is an
  $F$-stationary varifold with entropy $\lambda (\Sigma) \leq
  \lambda_0$ and $ d_V \left(\Sigma , \SS^k \times \RR^{n-k} \right)
  \leq \kappa$, then $\Sigma$ is isometric to $\SS^k \times \RR^{n-k}$.
\end{Cor}
  
Here $\SS^k \times \RR^{n-k}$ has multiplicity one and the $\SS^k$ factor 
is centered at $0$ with radius $\sqrt{2k}$.
  
\begin{proof}[Proof of Corollary \ref{c:uniq2}]
  This follows immediately from Theorem \ref{t:uniq1} since
  measure-theoretic closeness to a smooth shrinker implies smooth
  closeness on compact sets by Allard's regularity theorem \cite{Al}.
  (Note that we need the version of the Theorem that assumes
  smoothness only in $B_R$.)
\end{proof}

The rescaled mean curvature flow is obtained
from a MCF by setting $N_s=\frac{1}{\sqrt{t_0-t}}(M_t-x_0)$,
$s=-\log(t_0-t)$, $t<t_0$, where $(x_0,t_0)$ is some fixed point in
spacetime. It satisfies the equation $(\partial_s x)^\perp= -H\nu +
x/2$.

\begin{Pro} \label{p:compact} Given $n$, $\lambda_0$ and $\eps >
  0$, there exists $\delta > 0$ so that if $N_s \subset \RR^{n+1}$ is
  a rescaled MCF of integral varifolds for $s \in [0,1]$ with $\lambda
  (N_0) \leq \lambda_0$ and
  \begin{align} \label{e:compact}
     F(N_0) - F(N_1) \leq \delta \, ,
  \end{align}
  then there is an $F$-stationary varifold $\Sigma$ with
  $\lambda(\Sigma)\leq\lambda_0$ and $d_V (\Sigma , N_s) \leq \eps$ for
  all $s\in[0,1]$.
\end{Pro}

\begin{proof}
  We will argue as in the existence of tangent flows in \cite{I1}, \cite{W3}.  
  Suppose therefore that $N^i_s$ is a sequence of rescaled MCFs 
  satisfying \eqr{e:compact} with $\delta= 1/i$, but where  
  \begin{align} \label{e:fails} d_V (\Sigma , N^i_{s_i}) \geq
    \eps>0 {\text{ for some $s_i$ and every such }} \Sigma \, .
  \end{align}
  The Brakke compactness of Ilmanen (Lemma $7.1$ in \cite{I2}) gives a
  subsequence of the $N^i_s$'s converges to a limiting Brakke flow
  $N_s$ with $\lambda (N_0) \leq \lambda_0$ and $F(N_0) = F(N_1)$.  In
  particular, $N_s$ is a static solution of the rescaled MCF (i.e.,
  each $N_s$ is the same $F$-stationary varifold $\Sigma$). Since
  there is no mass loss, every sequence of time slices converges to
  $\Sigma$ as Radon measures, contradicting \eqr{e:fails}.
\end{proof}

\begin{proof}[Proof of Theorem \ref{t:main}]
  Let $\kappa > 0$ be given by Corollary \ref{c:uniq2}, and let
  $\delta>0$ by given by Corollary \ref{c:uniq2} with $\eps$
  replaced by $\kappa/6$.

  Let $(x_0,t_0)$ be a singular point with at least one tangent flow
  isometric to a cylinder $\SS^k\times\R^{n-k}$. Then there is a time
  $t<t_0$ such that $\frac{1}{\sqrt{t_0-t}}(M_t-x_0)$ is $C^2$ close
  to a cylinder in a large ball. By \eqr{e:elem1}, the gaussian
  integral is predictably small outside of this ball, so there is a
  space-time neighborhood $U$ of $(x_0,t_0)$ such that for any
  $(x_1,t_1) \in U$,
  \begin{align*} 
    d_V \left( \frac{1}{\sqrt{t_1-t}} \left( M_t - x_1 \right) ,
      O(\SS^k \times \RR^{n-k}) \right) \leq \frac{\kappa}{6} \, ,
  \end{align*}
  for some $O\in SO(n+1)$, and 
  \begin{align*}
    F \left(\frac{1}{\sqrt{t_1-t}}(M_t - x_1)\right) \leq \lambda_k +
    \frac{\delta}{2}\,,
  \end{align*}
  where $\lambda_k \equiv F (\SS^k \times \RR^{n-k})$.

  Now suppose that $(x_1,t_1)$ has entropy at least $\lambda_k -
  \delta/2$.  If $N_s$, $s\ge s_0$, is the rescaled MCF starting from
  $\frac{1}{\sqrt{t_1-t}}(M_t - x_1)$, then the total variation of
  $F(N_s)$ is at most $\delta$.  For each   integer $j \geq 0$, it follows by the choice of $\delta$
  that there is an $F$-stationary varifold $V_j$ so that 
  \begin{align}	\label{e:gonnatri}
      d_V (N_s , V_j) \leq \frac{\kappa}{6} {\text{ for every }} s \in [s_0 + j , s_0 + j +1] \, .
  \end{align}
  We have that $V_0$ is a generalized cylinder
  by Corollary \ref{c:uniq2}.  Furthermore, \eqr{e:gonnatri} and the triangle inequality
  implies that $d_V (V_j , V_{j+1}) \leq \kappa/3$, so every $V_j$ is a generalized cylinder,
  giving the theorem.
\end{proof}

\begin{proof}[Proof of Corollary \ref{c:axis}]
  The corollary follows immediately from Theorem \ref{t:main}.
\end{proof}

\section{Proof of Proposition \ref{p:step1}}		\label{s:step1}
 
In this section, we will prove Proposition \ref{p:step1}, which shows
that a curvature bound on a sufficiently large ball $B_R$ implies a
slightly worse bound on a larger ball $B_{(1+\theta)R}$ for some fixed
$\theta > 0$.  We will do this by using Brakke's theorem to prove
estimates in $B_R$ for the mean curvature flow $\Sigma_t = \sqrt{-t}
\, \Sigma$ associated to the shrinker.

\subsection{Applying Brakke to a self-shrinker}

We will use the following consequence of Allard's theorem, \cite{Al}:

\begin{Lem}	\label{l:allardreg}
Given $n$, there exists $\epsilon_A > 0$ so that if $\Sigma \subset \RR^{n+1}$ is an $F$-stationary varifold, $x_0 \in \Sigma$, and
there is some $\tau > 0$ so that
\begin{align}	\label{e:epsA}
	F_{x_0 , t_0} (\Sigma_0)  \leq 1 + \epsilon_A {\text{ for every }} t_0 \in (0, \tau) \, , 
\end{align}
then $\Sigma$ is smooth at $x_0$.
\end{Lem}

\begin{proof}
Since $\Sigma$ is $F$-stationary, it is stationary for a conformal metric.  Thus, by Allard's theorem,  it suffices to show that some tangent cone to $\Sigma$ at $x_0$ is a multiplicity one hyperplane.   

Any tangent cone $V$ must be a conical 
stationary integral varifold in $\RR^{n+1}$.  Since $V$ is a cone, we can compute $F(V)$ 
\begin{align}
	F(V) = (4\pi)^{ - \frac{n}{2} } \, \int_{0}^{\infty} 
	 \e^{ - \frac{s^2}{4} } \, s^{n-1} \, \Vol (\partial B_1 \cap V) \, ds 
	=\frac{ \Vol (\partial B_1 \cap V)}{\Vol (\partial B_1 \cap \RR^n)} =  \frac{ \Vol (  B_1 \cap V)}{\Vol ( B_1 \cap \RR^n)} \, .
\end{align}
Combining this with \eqr{e:epsA} implies that 
\begin{align}
 	\frac{ \Vol ( B_1 \cap V)}{\Vol ( B_1 \cap \RR^n)}  = F(V) \leq 1 + \epsilon_A \, .
\end{align}
Finally, Allard gives $\epsilon_A > 0$ so that if $V$ is any stationary varifold with 
$\frac{ \Vol ( B_1 \cap V)}{\Vol ( B_1 \cap \RR^n)}  \leq 1 + \epsilon_A$, then $V = \RR^n$ and has multiplicity one.

\end{proof}

\begin{Pro}	\label{p:brakke}
Given $\lambda_0 < 2$ and $C_1$, there exist $\theta > 0$ and $C_2$ 
so that if $\Sigma$ is $F$-stationary with
\begin{itemize}
\item $\lambda (\Sigma) \leq \lambda_0$,
\item $B_R \cap \Sigma$ is smooth and has $|A| \leq C_1$,
\end{itemize}
then $B_{(1+\theta)R-1/3} \cap \Sigma$ is smooth.  Furthermore, we get the curvature estimate
\begin{align}
    \sup_{B_{(1+\theta)R - 1/2} \cap \Sigma} |A| \leq C_2 \, .
\end{align}
\end{Pro}

\begin{proof}
We start with an almost-optimal bound on the $F$ functionals at small scales:
\begin{itemize}
\item[($\star$)]
Given any $\epsilon > 0$, there exists  $r_{\epsilon} \in (0,1)$ so that
\begin{align}
	F_{x_0 , t_0} (\Sigma_0)  \leq 1 + \epsilon
\end{align}
 for every $x_0 \in B_{R-1/4}$
and $t_0 \in (0, r_{\epsilon}^2)$.
\end{itemize}
\vskip1mm
Observe that the curvature bound implies that on any sufficiently small
scale $\Sigma$ decomposes into a collection of graphs with small gradient, but the entropy bound implies that
there is only one such graph.  Finally, (the scaled version of) Lemma \ref{l:elem} gives
($\star$).

The next step is to extend the entropy bound ($\star$) to a larger scale.  This follows from using Huisken's monotonicity for the
associated flow.  We will instead use the shrinker equation and, in particular, 
the following result from section $7$ in \cite{CM1} (see equation ($7.13$) there):

Given $y \in \RR^{n+1}$, $a \in \RR$, and $s > 1$ so that $1+a s^2 > 0$, then
\begin{align}	\label{e:yas}
	F_{sy, 1+as^2} (\Sigma) \leq F_{y,1+a} (\Sigma)  \, .
\end{align} 
We apply this to each $y \in B_{R-1}$ with $a=t_0 -1$ (for each $t_0$).  This can be done so long 
\begin{align}
	1 + s^2 (t_0 -1) = 1 + a s^2 > 0 \, .
\end{align}
It follows that we get some $\theta > 0$ and $\tau > 0$ so that every $x_0 \in B_{(1+\theta)(R-1/4)}$ has
\begin{align}
	F_{x_0 , t_0} (\Sigma) \leq 1 +\epsilon {\text{ for all }} t_0 \leq \tau \, .
\end{align}
Regularity now follows from Lemma \ref{l:allardreg}.

Since we have ($\star$), the curvature estimate now follows from
applying theorem $3.1$ in \cite{W2} to the associated MCF
 $\Sigma_t \equiv \sqrt{-t} \, \Sigma$.  We apply \cite{W2} on each ball $B_{r_{\epsilon}}(y)$ for 
$y \in B_{R-1/2}$ to get a curvature bound on this ball at time $r_0^2-1$.  This yields the claimed 
 $|A|$ bound on $\Sigma$ on the larger ball.
\end{proof}

\subsection{Extending curvature bounds outward}

 We can now   prove the first main ingredient for the rigidity.

\begin{proof}[Proof of Proposition \ref{p:step1}]
Proposition \ref{p:brakke}  immediately extends smoothness and the $|A|$ bound to a larger scale, but 
it remains to extend the positivity of $H$ to the larger scale.  This will follow from getting a uniform
bound on the time derivative of $H$ for the MCF
  $\Sigma_t \equiv \sqrt{-t} \, \Sigma$   corresponding to $\Sigma$.  This bound extends the positivity of
  $H$ forward in time for $\Sigma_t$ which corresponds to positivity of $H$ on the larger scale for $\Sigma$.

 We first use Proposition \ref{p:brakke}  to get   $\delta > 0$ and $C$ so that
 \begin{align}	\label{e:33}
    \sup_{B_{R} \cap \Sigma_t} |A| \leq C {\text{ for }} t \in (-1, \delta -1) \, .
 \end{align}
   Parabolic estimates of \cite{EH} (applied on balls of unit scale) 
   then give the uniform higher derivative bounds 
 \begin{align}
    \sup_{B_{R-1} \cap \Sigma_t} |\nabla A| + \left| \nabla^2 A \right|
    \leq C' {\text{ for }} t \in (-1, \delta -1) \, .
 \end{align}
 Using the equation $\partial_t H = \Delta H + |A|^2 \, H$, these bounds on $A$ and its derivatives give
 a uniform bound on 
  $\partial_t H$ on this set and, thus, strict positivity of $H$ propagates
 forward in time for some positive time interval,
completing the proof.
 
 \end{proof}

\section{Proof of Proposition \ref{p:step2}}  \label{s:step2}

In this section, we will prove the second key ingredient (Proposition
\ref{p:step2}) that gives the improvement in the cylindrical
estimates.  This proposition shows that if a shrinker has slightly
positive mean curvature and some large curvature bound on a large
ball, then we get much better bounds on $H$ and $A$ on a slightly smaller
ball.  Crucially, these uniform bounds do not depend on the initial
radius, so we get a fixed improvement that can be iterated.

\vskip2mm
The proof of Proposition \ref{p:step2} uses the positivity of $H$ to
prove that the tensor $\tau \equiv \frac{A}{H}$ is almost parallel.
This is proven over the next two subsections.  We then show that the
eigenvalues of $\tau$ fall into two clusters, one near zero
corresponding to the translation invariant directions of the
generalized cylinder and the other corresponding to the spherical part
of the cylinder.  This is done in the last two subsections.

\vskip2mm
Throughout this section, $\Sigma \subset \RR^{n+1}$ will be an
$n$-dimensional $F$-stationary varifold.

\subsection{Simons' equation}

This subsection contains some calculations that will be used to show
that $\tau = A/H$ is almost parallel.  The calculations in this
section are purely local and, thus, are valid at any smooth point of
$\Sigma$.

Given $f>0$, define a weighted divergence operator $\dv_f$ and drift
Laplacian $\cL_f$ by
\begin{align}
  \dv_f (V) &= \frac{1}{f} \, \e^{ |x|^2/4 } \, \dv_{\Sigma}
  \, \left( f \, \e^{ -|x|^2/4 } \, V \right) \, ,
  \\
  \cL_f \, u &\equiv \dv_f (\nabla u) =\cL \, u + \langle \nabla \log
  f , \nabla u \rangle \, .
\end{align}
Here $u$ may also be a tensor; in this case the divergence traces only
with $\nabla$. Note that $\cL=\cL_1$. We recall the quotient rule:
\begin{Lem} \label{l:quotr} Given a tensor $\tau$ and a function $g$
  with $g \ne 0$, then
  \begin{align}
    \cL_{g^2} \, \frac{\tau}{g} & = \frac{ g \, \cL \, \tau - \tau \,
      \cL \, g}{g^2} = \frac{ g \, L \, \tau - \tau \, L \, g}{g^2}
    \,.
  \end{align}
\end{Lem}
 
\begin{proof}
  We compute
  \begin{align}
    \cL \, \frac{\tau}{g} &= \e^{ |x|^2/4 } \, \dv_{\Sigma} \,
    \left( \e^{ -|x|^2/4 } \, \nabla \frac{\tau}{g} \right) =
    \e^{ |x|^2/4 } \, \dv_{\Sigma} \, \left( \e^{
        -|x|^2/4 } \, \left[ \frac{\nabla \tau}{g} -
        \frac{\tau \nabla g}{g^2} \right]
    \right) \notag \\
    & = \frac{ \cL \, \tau}{g} - \frac{\tau \, \cL \, g}{g^2} - 2\,
    \frac{ \langle \nabla \tau , \nabla g \rangle }{g^2} + 2 \,
    \frac{\tau \, |\nabla g|^2}{g^3} = \frac{ \cL \, \tau}{g} - \frac{
      \tau \, \cL \, g}{g^2} - \langle \nabla \, \log g^2 , \nabla \,
    \frac{\tau}{g} \, \rangle \notag \, .
  \end{align}
\end{proof}
 
\begin{Pro} \label{p:eqnsg} On the set where $H > 0$, we have
  \begin{align}
    \cL_{ H^2} \, \frac{ A}{H} &=   0 \, , \\
    \cL_{ H^2} \, \frac{ |A|^2}{H^2} &= 2\, \left|\nabla \frac{A}{H}
    \right|^2 \, .
  \end{align}
\end{Pro}

\begin{proof}
  Since $L H = H$ and $L A = A$ by \cite{CM1}, the first claim follows
  from the quotient rule (Lemma \ref{l:quotr}).  The second claim
  follows from the first since $\frac{ |A|^2}{H^2} = \langle \frac{
    A}{H} , \frac{ A}{H} \rangle$.
\end{proof}

\subsection{An effective bound for shrinkers with positive mean curvature}

The next proposition gives exponentially decaying integral bounds for
$\nabla (A/H)$ when $H$ is positive on a large ball.  It will be
important that these bounds decay rapidly.

\begin{Pro} \label{p:effective} If $B_R \cap \Sigma$ is smooth with $H
  > 0$, then for $s \in (0 , R)$ we have
  \begin{align}
    \int_{B_{R-s} \cap \Sigma} \left| \nabla \frac{A}{H} \right|^2 \,
    |H|^2 \, \e^{- |x|^2/4 } &\leq \frac{4}{s^2} \, \sup_{B_R
      \cap \Sigma} |A|^2 \, \Vol (B_R \cap \Sigma) \, \e^{ -
      \frac{(R-s)^2}{4} } \, .
  \end{align}
\end{Pro}

\begin{proof}
  Set $\tau= A/H$ and $u = | \tau |^2 = |A|^2/H^2$, so that $\cL_{H^2}
  \, u = 2\, \left| \nabla \tau \right|^2$ by Proposition
  \ref{p:eqnsg}.  Fix a smooth cutoff function $\phi$ with support in
  $B_R$. Using the divergence theorem, the formulas for
  $\cL_{H^2}$, and the absorbing inequality $4ab \leq a^2 + 4b^2$, we get
  \begin{align}
    0 &= \int_{B_R \cap \Sigma} \dv_{H^2} \, \left( \phi^2 \, \nabla u \right) \, H^2 \, \e^{- |x|^2/4 } \notag \\
    &= \int_{B_R \cap \Sigma} \left(\phi^2 \, \cL_{H^2}u + 4 \phi
      \langle \nabla \phi , \nabla u \rangle \right)
    \, H^2 \, \e^{- |x|^2/4 }   \\
    &= \int_{B_R \cap \Sigma} \left(2\, \phi^2 \, \left| \nabla \tau
      \right|^2 + 4 \phi \langle \nabla \phi , \tau\cdot\nabla \tau
      \rangle \right)
    \, H^2 \, \e^{- |x|^2/4 } \notag  \\
    &\geq \int_{B_R \cap \Sigma} \left(\phi^2 \, \left| \nabla \tau
      \right|^2- 4 \, |\tau|^2 \, \left| \nabla \phi \right|^2 \right) \,
    H^2 \, \e^{- |x|^2/4 }, \notag
  \end{align}
  from which we obtain
  \begin{equation}
    \int_{B_R \cap \Sigma} \left(\phi^2 \, \left| \nabla \frac{A}{H}
      \right|^2    \, |H|^2 \right) \,   \e^{- |x|^2/4 } 
    \leq 4\, \int_{B_R \cap \Sigma} \left| \nabla \phi \right|^2 \, |A|^2 \, 
    \e^{- |x|^2/4}.
  \end{equation}
  The proposition follows by choosing $\phi \equiv 1$ on $B_{R-s}$ and
  going to zero linearly on $\partial B_{R}$.
\end{proof}

The next corollary establishes bounds on two derivatives of the tensor
$\tau = A/H$ by combining the integral estimates of the
previous proposition with elliptic estimates.

\begin{Cor} 
  \label{c:taubound} 
  Given $n>0$, $\delta>0$, $\lambda_0>0$ and $C_1>0$ there exists a
  constant $C_{\tau}>0$ such that if $\lambda (\Sigma) \leq
  \lambda_0$, $R\geq 2$, and
  \begin{itemize}
  \item $B_R \cap \Sigma$ is smooth with $|A| \leq C_1$ and $H\geq
    \delta>0$,
  \end{itemize}
  then 
  \begin{align}
    \sup_{ B_{R-2} \cap \Sigma } \, \, \left| \nabla \tau \right|^2 +
    R^{-2} \, \left| \nabla^2 \tau \right|^2 \leq C_{\tau} \, R^{2n}
    \, \e^{-R/4} \, .
  \end{align}
\end{Cor}

\begin{proof}
  Throughout this proof $C$ will be a constant that depends only on
  $n$, $\delta$, $\lambda_0$ and $C_1$; $C$ will be allowed to change
  from line to line.

  Proposition \ref{p:effective} with $s=1/2$ gives
  \begin{align}
    \int_{B_{R-1/2} \cap \Sigma} \left| \nabla \tau \right|^2 \, \e^{-
      |x|^2/4 } &\leq C \, R^n \, \e^{ - (R-1/2)^2/4} \, .
  \end{align}
  Since $\e^{- |x|^2/4 } \geq \e^{- \frac{R^2 -2R + 1}{4} } $ on
  $B_{R-1}$, it follows that
  \begin{align}
    \int_{B_{R-1} \cap \Sigma} \left|\nabla \tau \right|^2 &\leq C \,
    R^n \, \e^{ - \frac{R}{4} } \, .
  \end{align}
  This gives the desired integral decay on $\nabla \tau$.  We will
  combine this with elliptic theory to get the pointwise bounds. The
  key is that $\tau$ satisfies the elliptic equation $\cL_{H^2} \tau =
  0$.  The two first order terms in the equation come from $x^T$ in
  $\cL$ and $\nabla H$; both grow at most linearly (in the second
  case, we differentiate the shrinker equation and use the bound on $|A|$).
  Therefore, we can apply elliptic theory on balls of radius $1/R$ to
  get for any $p \in B_{R-2} \cap \Sigma$ that
  \begin{align}
    \left( |\nabla \tau|^2 + R^{-2} \, | \nabla^2 \tau |^2 \right)(p)
    \leq C \, R^n \, \int_{B_{\frac{1}{R} } (p) \cap \Sigma} | \nabla
    \tau |^2 \, .
  \end{align}
  Combining this with the integral bound gives the corollary.
\end{proof}

\subsection{Finding small eigenvalues of $A$}
  
The next lemma shows that if we have an almost parallel symmetric
$2$-tensor with two distinct eigenvalues, then the plane spanned by
the corresponding eigenvectors is almost flat.
  
The results in this subsection are valid for a general smooth
hypersurface, possibly with boundary, and do not use an equation for
the hypersurface.

\begin{Lem} \label{l:curvature} Suppose that $B$ is a symmetric
two-tensor on a Riemannian manifold with $|\nabla B| , |\nabla^2 B|
\leq \eps\leq 1$. If $e_1,e_2$ are unit vectors at $p$
with
\begin{align*}
  B_p (e_1) = 0 {\text{ and }}  B_p (e_2) = \kappa \, e_2,\qquad \kappa\ne0,
\end{align*}
then the sectional curvature $K$ of the 2-plane spanned by $e_1$ and
$e_2$ satisfies
\begin{align*}
  \left| K \right| \leq   2\, \eps \, \left( \frac{1}{|\kappa|} + \frac{1}{\kappa^2} 
    \right) \, .
\end{align*}
\end{Lem}
 
\begin{proof}
Let $v_1$ and $v_2$ be smooth vector fields in a neighborhood of $p$
with 
$$
v_j (p) = e_j,\qquad \nabla v_j(p)=0,\qquad |v_j| \leq 1,
$$ 
and set
\begin{align*}
  w_2 \equiv \frac{B(v_2)}{\kappa}
\end{align*}
so $w_2 (p) = e_2 (p)$.  Then the sectional curvature $K$ at $p$ is
\begin{align}	\label{curv}	
   K = \langle R(e_2 , e_1) e_2 , e_1 \rangle = 
    \langle \nabla_{v_1} \nabla_{w_2} w_2 , v_1 \rangle -
   \langle \nabla_{w_2} \nabla_{v_1} w_2 , v_1 \rangle - 
  \langle \nabla_{[v_1 , w_2]} w_2 , v_1 \rangle \, .
\end{align}
We will bound each of these terms by showing that $\nabla w_2$ and
$\nabla^2 w_2$ are orthogonal to $v_1$ up to small error terms.
    
Let $x$ and $y$ be vector fields with $|x(p)| = |y(p)| = 1$.  We have
\begin{align*}
  \kappa  \, \nabla_x w_2 =  (\nabla_x B) (v_2)  +  B (\nabla_x v_2)   \, .
\end{align*} 
so at $p$ one has
\begin{align}
  |\nabla_x w_2(p)|\le\frac{\eps}{|\kappa|},\quad
  |[v_1,w_2](p)|\le\frac{\eps}{|\kappa|},\quad
  |\langle \nabla_{[v_1 , w_2]} w_2 , v_1 \rangle(p)|
  \leq\frac{\eps^2}{|\kappa|^2},
\label{a1}
\end{align}
thus estimating the third term of \eqr{curv}.

Differentiating again, we have
\begin{align*}
  \kappa\,\nabla_y\nabla_x w_2 &= \nabla_y \left\{(\nabla_x B) (v_2)  +
    B (\nabla_x v_2)  \right\} \\
   &= (\nabla^2_{y,x} B) (v_2)  +(\nabla_{ \nabla_y x}B) (v_2)+(\nabla_x B) (\nabla_y v_2) \\
  &\qquad + (\nabla_y B) (\nabla_x v_2) + B (\nabla_y \nabla_x v_2)\,. 
\end{align*}
The last term on the right is in the range of $B$ and, thus, is
orthogonal to $v_1(p) = e_1$ which is in the kernel of $B$.  
We obtain at $p$
\begin{align}
\begin{aligned}
|\langle\nabla_{v_1} \nabla_{w_2} w_2,v_1\rangle(p)|
&\leq  \frac{1}{|\kappa|}\bigg(|\langle \nabla^2_{v_1,w_2} B)
(v_2),v_1\rangle(p)| \\
&\qquad\qquad\qquad\qquad+|\langle\nabla_{\nabla_{v_1}w_2} B)
(v_2),v_1\rangle(p)|+0+0\bigg)\\
&\leq\frac{1}{|\kappa|}\left(\eps+\frac{\eps^2}{|\kappa|}\right)
=\frac{\eps}{|\kappa|}+\frac{\eps^2}{|\kappa|^2}.
\end{aligned}
\label{a2}
\end{align}
Similarly we obtain
\begin{align}
|\langle\nabla_{w_2} \nabla_{v_1} w_2,v_1\rangle(p)|
\leq\frac{\eps}{|\kappa|}.
\label{a3}
\end{align}
Combining \eqr{curv}, \eqr{a1}, \eqr{a2} and \eqr{a3} gives the desired bound.
\end{proof}

We will apply the above lemma when $B$ is given by $\tau - \kappa_1 \,
g$ where $\kappa_1$ is one of the small eigenvalues of $\tau$ at the
point $p$.

\begin{Cor} \label{c:lowev} Suppose that $\Sigma \subset \RR^{n+1}$ is
  a hypersurface (possibly with boundary) with the following
  properties
 \begin{itemize}
  \item $0 <  \delta \leq H$ on $\Sigma$.
  \item The tensor $\tau \equiv A/H$ satisfies $\left| \nabla \tau
    \right| + \left| \nabla^2 \tau \right| \leq \eps\leq 1$.
  \item At the point $p \in \Sigma$, $\tau_p$ has at least two distinct eigenvalues $\kappa_1 \ne \kappa_2$.
 \end{itemize}
Then  
\begin{align*}
  \left| \kappa_1 \kappa_2 \right| \leq \frac{2 \, \eps}{\delta^2}
  \, \left(  \frac{1}{\left| \kappa_1 - \kappa_2 \right|}  +
  \frac{1}{\left| \kappa_1 - \kappa_2 \right|^2} \right) \, .
\end{align*}

\end{Cor}

\begin{proof}
Let $K_p$ be the sectional curvature of the plane at $p$ spanned by an
eigenvector for $\tau_p$ with eigenvalue $\kappa_1$ and one with eigenvalue $\kappa_2$.  Since 
$A= H \, \tau$, the Gauss 
equation gives  
\begin{align*}
  |K_p| = H^2 \, \left| \kappa_1 \kappa_2 \right| \geq \delta^2 \, \left| \kappa_1 \kappa_2 \right|  \, .
\end{align*}
On the other hand, Lemma \ref{l:curvature} with $B = \tau - \kappa_1
\, g$ and $\kappa = \kappa_2 - \kappa_1$ gives  
\begin{align*}
   |K_p| \leq  2\, \eps \, \left( \frac{1}{ |\kappa_2
     - \kappa_1|} + \frac{1}{ |\kappa_2
     - \kappa_1|^2} 
    \right) \, .
\end{align*}
The corollary follows by combining these.
\end{proof}

\subsection{The proof of Proposition \ref{p:step2}}

We are now ready to prove the second key proposition.

\begin{proof}[Proof of Proposition \ref{p:step2}]
Fix $n$, $\lambda_0>0$, $\delta_0>0$, and $C_0>0$. Let $R>0$ and
assume that $\Sigma$ is a shrinker in $\R^{n+1}$,
$\lambda(\Sigma)\le\lambda_0$, $\Sigma\cap B_R$ is smooth and
$$
|A| \leq C_0\text{ and }0 < \delta_0 \leq H\qquad\text{ on }\Sigma\cap
B_R.
$$
It follows by Corollary \ref{c:taubound} that the tensor $\tau \equiv
A/H$ satisfies
\begin{align*}
\left| \nabla \tau \right| + \left|\nabla^2 \tau \right| 
\leq \eps_{\tau}\quad\text{ on }B_{R-2} \cap \Sigma,
\end{align*}
where
\begin{align*}
  \eps_{\tau}^2 := C \, R^{2n+2} \, \e^{-R/4} 
\end{align*}
and the constant $C$ depends only on $n$, $C_0$, $\delta_0$ and
$\lambda_0$ (and, in particular, not on $R$).

Now fix some small $\eps_0 > 0$, to be reduced as needed, but
depending only on $n$. Combining the compactness result of Lemma
\ref{l:cpns} with Huisken's classification (\cite{H1} and \cite{H2})
of complete shrinkers with $H \geq 0$ and bounded $|A|$, there
exists $R_1>0$ depending on $\lambda_0$, $C_0$, $n$, and $\eps_0$
such that if $R\ge R_1$, then for some $k\in\{1,\ldots n\}$,
$B_{5\sqrt{2n}} \cap \Sigma$ is $C^2$ $\eps_0$-close to a cylinder
$\SS^k \times \RR^{n-k}$ where $\SS^k$ has radius $\sqrt{2k}$. In
particular, we can arrange that
\begin{align}
|A|^2 \leq 3/4\text{ and }1/2 \leq H\quad\text{ on }B_{5\sqrt{2n}} \cap \Sigma,
\label{nearly}
\end{align}
and further, that at every $p$ in $\Sigma\cap B_{5\sqrt{2n}}$ there
are $n-k$ orthonormal eigenvectors
\begin{align*}
v_1 (p) , \ldots, v_{n-k} (p),
\end{align*}
of $A$ with eigenvalues less than $1/\sqrt{100n}$, plus at least one
eigenvector with eigenvalue at least $1/\sqrt{4n}$. Then we can apply
Corollary \ref{c:lowev} to obtain
\begin{align*}
  \left|\kappa_j(p)\right|\leq C \, \eps_{\tau},\qquad j=1,\ldots,n-k,
\end{align*}
where $C$ depends only on $n$, $C_0$, $\delta_0$ and $\lambda_0$.
 
Now fix $p$ in $\Sigma\cap B_{2\sqrt{2n}}$ and define $n-k$
tangential vector fields $v_i$ on $\Sigma$ by
\begin{align*}
  v_i = v_i (p) - \langle v_i (p) , \nn \rangle \, \nn \, .
\end{align*}
Fix a constant $L$ (later we will take $L=2R$) and let $\Omega$ denote
the set of points in $B_{R-2} \cap \Sigma$ that can be reached from $p$ by a
path in $B_{R-2}\cap \Sigma$ of length at most
$L$.

We will show that the $v_i$'s have the following three properties on $\Omega$:
\begin{align}
\left| v_i-v_i(p)\right|&\leq C\,(L+L^2)\,\eps_{\tau}\,,\label{e:v1}  \\
\left| \tau (v_i )\right|&\leq C\,(1+L^2)\,\eps_{\tau}\,,\label{e:v2} \\
\left|  \nabla_{v_i} A \right|&\leq C\,(1+L)(1+L^2)\,\eps_{\tau}\,.
\label{e:v3} 
\end{align}
To prove \eqr{e:v1} and \eqr{e:v2}, suppose that $\gamma:[0,L] \to
\Sigma$ is a curve with $\gamma(0) = p$ and $|\gamma' | \leq 1$ and
that $w$ is a parallel unit vector field along $\gamma$ with $w(0) =
v_i (p)$.  Therefore, the bound on $\nabla \tau$ gives $\left|
\nabla_{\gamma'} \, \tau (w) \right| \leq \eps_{\tau}$ and, thus,
\begin{align*}
  \left| \tau (w)\right|\leq L\,\eps_{\tau}+\left|\tau_p (v_i (p)) \right| 
  \leq (C+L ) \, \eps_{\tau} \, .
\end{align*}
In particular, we also have
\begin{align*}
  \left| A (w) \right|= |H| \, \left| \tau (w) \right| 
  \leq C\, (C+L ) \, \eps_{\tau} \, .
\end{align*}
Therefore, since $\nabla_{\gamma'}^{\RR^{n+1}} w = A (\gamma' , w)$,
the fundamental theorem of calculus gives
\begin{align}	\label{e:keyg}
  \left| w(t) - v_i (p) \right| = \left| w(t) - w(0) \right| \leq C\,
  (L+L^2 ) \, \eps_{\tau} \, .
\end{align}
Since $\langle w(t),\nn\rangle=0$, we see that $\left| \langle v_i(p),
\nn \rangle \right| \leq C\, (L+L^2 ) \, \eps_{\tau}$, giving
\eqr{e:v1}.  Adding \eqr{e:v1} and \eqr{e:keyg} yields
\begin{align*}
  \left| w(t) - v_i \right| \leq C\, (L+L^2 ) \, \eps_{\tau} \, ,
\end{align*}
so we get that
\begin{align*}
  \left| \tau (v_i) \right| \leq \left| \tau (w) \right| + \left| \tau
  \left( w - v_i \right) \right| \leq C\, (1+L^2 ) \, \eps_{\tau}
\end{align*}
which is \eqr{e:v2}. 
   
Next we will see that  \eqr{e:v2} implies \eqr{e:v3}.  Namely,
the  Codazzi equation gives  
\begin{align*}
  \left| \left( \nabla_{v_i} A \right) (x,y) \right| &= 
  \left| \left( \nabla_{x} A \right) (v_i,y) \right| = 
  \left| \left( \nabla_{x} (H \, \tau) \right) (v_i ,y) \right| \notag \\
  &\leq \left|H  \left( \nabla_{x}  \tau \right) (v_i,y) \right|
  + \left| \left( \nabla_{x} H \right)\, \tau (v_i,y)\right| \\
  &\leq C\,  \eps_{\tau} + C \,(1+L)(1+ L^2) \,  \eps_{\tau} \, ,
\end{align*}
where we have used the linear growth estimate $|\nabla H|\leq
1+R|A|\le 1+CL$ as in the proof of Corollary
\ref{c:taubound}. This gives \eqr{e:v3}.

The key for the bounds \eqr{e:v1}--\eqr{e:v3} is that
$\eps_{\tau}$ decays exponentially in $R$ and, thus, these bounds
are extremely small even compared to any power of $R$.
 
Finally, \eqr{e:v1} and \eqr{e:v3} together allow us to extend the
nearly sharp bounds on $A$ and $H$ in $B_{5\sqrt{2n}} \cap \Sigma$
given by \eqr{nearly} to all of $\Omega$. Define the linear functions
\begin{align*}
  f_i (x) = \langle v_i (p) , x \rangle \,, \qquad i=1,\ldots,n-k.
\end{align*}
From the $\eps_0$ closeness to the cylinder in the ball
$B_{5\sqrt{2n}}$, we know that
\begin{align*}
  \Sigma_0 \equiv B_{5\sqrt{2n}} \cap \Sigma \cap \{ f_1 = \cdots =
  f_{n-k} = 0 \}
\end{align*}
is a compact topological $\SS^k$ of radius approximately $\sqrt{2k}$.

When $n-k=1$ (there is only one approximate translation), then the
flow of the vector field $v_1/|v_1|^2$ starting from $\Sigma_0$
preserves the level sets (in $\Sigma$) of $f_1$, the flow vector field
nearly equals the constant vector $v_i(p)$ by \eqr{e:v1}, and the
second fundamental form $A$ is almost parallel along the flow lines,
for as long as the flow remains within $\Omega$.  Note that if
$R>30n$, then the disk $D$ of radius $\sqrt{2n}+1$ in the plane
$f_1=R-3$ with center the point closest to the origin lies within
$B_{R-2}$. Also, by considering the shape of the approximate cylinder,
the path distance within $\Sigma$ between $p$ and a point $x(t)$ on
the flow line is no more than the length of the flow line plus
$(5+2\pi)\sqrt{2n}+1$. So if we take $L=2R$ and $R^3\eps_\tau$ small
enough, then each flow line hits $D$ before it leaves $\Omega$.
Similarly, if we flow along the negative vector field, then again each
flow line hits $-D$ before it leaves $\Omega$.  Thus, the connected
component of $B_{R-3} \cap \Sigma$ containing $\Sigma_0$ is contained
in $\Omega$ and in the union of the flow lines.  In particular, the
nearly sharp bounds on $A$ and $H$ in $B_{5\sqrt{2n}} \cap \Sigma$
extend to this component. By the maximum principle, every component of
$B_{R-3} \cap \Sigma$ meets $\overline{B}_{\sqrt{2n}}$.  However,
there is only one such component since $B_{5\sqrt{2n}} \cap \Sigma$ is
cylindrical. Thus, the nearly sharp bounds are valid on $B_{R-3}\cap
\Sigma$, and the proof is complete for this case.
  
We argue similarly when $n-k > 1$, except that we construct a
single ``radial flow'' rather than doing each flow successively as
this seems simpler.  First, define $f$ by
\begin{align*}
  f^2 = \sum_{i=1}^{n-k} f_i^2 \, ,
\end{align*}
and then let 
\begin{align*}
  v = \frac{ \nabla f}{\left| \nabla f \right|^2}\,,
\end{align*}
where as before $\nabla$ is the tangential derivative, 
Thus,  the flow by $v$ preserves the level sets of $f$.  Note that
\begin{align*}
  \nabla f = \frac{ \sum f_i \nabla f_i}{f} = \sum \frac{f_i}{f} \, v_i\,.
\end{align*}
Since the $v_i
(p)$'s are orthonormal and each $v_i - v_i (p)$ is small by
\eqr{e:v1}, it follows that $\left| \nabla f \right|$ is almost one.
Furthermore, $v$ is given at each point as a linear combination of the
$v_i$'s, so \eqr{e:v3} implies that $\left| \nabla_{v} A \right|$ is
small.  Arguing as in the previous case completes the proof.

\end{proof}

\end{document}